\documentclass[12pt]{amsart}
\usepackage{amsmath,amsthm,amsfonts,amssymb,amscd,euscript}

\newlength{\defbaselineskip}
\theoremstyle{plain}

\theoremstyle{definition}

\theoremstyle{plain}
\newtheorem{thm}{Theorem}

\newtheorem{lem}[thm]{Lemma}

\newtheorem{conj}[thm]{Conjecture}

\theoremstyle{definition}
\newtheorem{defn}[thm]{Definition}

\newtheorem{exmp}[thm]{Example}

\newtheorem{rem}[thm]{Remark}
\numberwithin{equation}{section}
\begin{document}





\title{A step towards the cluster positivity conjecture}
\author{Kyungyong Lee}
\thanks{Research partially supported by NSF grant DMS 0901367.}


\address{Department of Mathematics, University of Connecticut, Storrs, CT 06269}
\email{{\tt kyungyong.lee@uconn.edu}}

 \maketitle
 



\section{introduction}
Let $K = k(x,y)$ be the skew field of rational functions in the non-commutative
variables $x$ and $y$, where the ground field $k$ is $\mathbb{Q}$, $\mathbb{Q}(q)$ or any field containing either one. For any positive integer $r$, let $F_r$ be the Kontsevich automorphism of $K$, which is defined by
\begin{equation}\label{Kont_map}F_r : \left\{\begin{array}{l}x \mapsto xyx^{-1} \\ y \mapsto (1+y^r)x^{-1}.\end{array}\right. \end{equation}
Let
$$x_n=F_r^n(x)=\begin{array}{c}\underbrace{F_r\circ F_r \circ \cdots \circ F_r}(x)\\n\,\,\,\,\,\,\,\,\,\end{array}\text{ and } y_n =F_r^n( y),$$
for every integer $n$.

\begin{conj}[Kontsevich]\label{Kontconj}
Both $x_n$ and $y_n$ are non-commutative Laurent polynomials of $x$ and $y$ with non-negative integer coefficients.
\end{conj}

Usnich \cite{U} showed, by using derived categories, that $x_n$ and $y_n$ are non-commutative Laurent polynomials of $x$ and $y$. Independently, Berenstein and Retakh found an elementary proof for the Laurent phenomenon in \cite{BR} (with $F_r^n$ replaced by $\cdots F_{r_1}F_{r_2}F_{r_1}$). The Kontsevich conjecture was verified by Di Francesco and Kedem for $(r_1, r_2) \in \{(2, 2), (4, 1), (1, 4)\}$ in \cite{DK}, with explicit combinatorial expressions. The case of $r_1r_2\leq 3$ is elementary. 

Conjecture~\ref{Kontconj} is related to the positivity conjecture for commutative cluster algebras. Cluster algebras were invented by Fomin and Zelevinsky \cite{FZ} in the Spring of 2000 as a tool for studying dual canonical bases and total positivity in semisimple
Lie groups. Since then, they have found applications and connections in a wide range of subjects including quiver representation theory, Poisson geometry, Teichm\"uller theory, tropical geometry and so forth. A cluster algebra is a commutative $\mathbb{Q}$-algebra  with a family of distinguished generators (the \emph{cluster variables}) grouped into overlapping subsets (the \emph{clusters}) which are constructed by mutations. One of major results in the theory of cluster algebras is the Laurent phenomenon of cluster variables, which was proved by Fomin and Zelevinsky in \cite{FZ}. Their theorem says that any cluster variable is a Laurent polynomial of cluster variables in any cluster, with integer coefficients. In addition, they conjectured that the coefficients are non-negative.    We mention some of known cases. For cluster algebras of finite type, the positivity conjecture with respect to a bipartite seed follows from a result of Fomin and Zelevinsky \cite[Corollary 11.7]{FZ4}. Musiker, Schiffler, and Williams \cite{MSW} found combinatorial formulas for the Laurent expansion of any cluster variable in any cluster algebra coming from a triangulated surface (with or without punctures) with respect to an arbitrary seed, which imply positivity of the Laurent expansion. For acyclic cluster algebras, Caldero and Reineke \cite{CR} made a significant progress, and Qin \cite{Q} and independently Nakajima \cite{N} who used the idea of Hernandez and Leclerc \cite{HL} proved the positivity conjecture for the special case of an initial seed. For rank 2 cluster algebras, Dupont \cite{D} showed the positivity by using the results for acyclic cluster algebras. For acyclic rank 2 cluster algebras, neither elementary proof nor combinatorial formula for the positive coefficients has been known except for the case of $r=2$ \cite{SZ}, \cite{MP}, \cite{CZ}.

In this note we prove Conjecture~\ref{Kontconj}.  Actually we show that for $r\geq 2$ and $n>0$, $x_n$ is a non-commutative Laurent polynomial of $x$ and $y$ with non-negative integer coefficients. The case of $n<0$ can be taken care of by considering the anti-automorphism of $k\langle x_1^{\pm1}, x_2^{\pm1}\rangle$ as in \cite[Lemma 7]{BR}. Since $F_r^{-1}$ is obtained by interchanging $x$ and $y$ in (\ref{Kont_map}), the statement for $y_n$ easily follows from the one for $x_n$. Once we establish an explicit formula for $x_n$ (see Theorem~\ref{mainthm}), proofs are straightforward. Furthermore, it is not hard to show that the formula implies that every nonzero coefficient is equal to 1, hence it gives a combinatorial formula for the specialization of $x_n$ at $xy=q^eyx$ for any integer $e$ including 0. Our proof is completely elementary.

\noindent\emph{Acknowledgement.} We are grateful to Professors Sergey Fomin and Rob Lazarsfeld for their valuable advice. We also thank Professors Philippe Di Francesco, David Hernandez, Vladimir Retakh and Ralf Schiffler for their useful suggestions. Retakh pointed out that the argument also works for $r=2$ and suggested to compare with known formulas. Di Francesco clarified the definition of the Kontsevich automorphism and Definition 3 and 5.

\section{Main Result}

Throughout the notes, we fix $r\geq 2$ and let $F=F_r$. We will always use the following presentation for monomials in $K$:
$$y^{\beta_0}x^{\alpha_1}y^{\beta_1}x^{\alpha_2}y^{\beta_2}\cdots  x^{\alpha_{m-1}}y^{\beta_{m-1}} x^{\alpha_m}\longleftrightarrow \left(\begin{array}{cccccc} \text{ }&\alpha_1 & \alpha_2 & \cdots & \alpha_{m-1}& \alpha_m\\ \beta_0 & \beta_1 & \beta_2 & \cdots & \beta_{m-1} & \text{ } \end{array} \right).$$
These expressions are not necessarily minimal, i.e. some of $\alpha_i$ or $\beta_i$ are allowed to be zeroes. As a matter of fact, zeroes serve as part of a potential for cancellation with $(1+y^r)^{-1}$, hence  we almost never use minimal presentations in these notes. The multiplication of two monomials are given in the obvious way:
\tiny{$$\left(\begin{array}{ccccc} \text{ }&\alpha_1  & \cdots & \alpha_{m-1}& \alpha_m\\ \beta_0 & \beta_1  & \cdots & \beta_{m-1} & \text{ } \end{array} \right)\left(\begin{array}{ccccc} \text{ }&\alpha_{m+1}  & \cdots & \alpha_{m+l-1}& \alpha_{m+l}\\ \beta_m & \beta_{m+1}  & \cdots & \beta_{m+l-1} & \text{ } \end{array} \right)=\left(\begin{array}{ccccc} \text{ }&\alpha_1  & \cdots & \alpha_{m+l-1}& \alpha_{m+l}\\ \beta_0 & \beta_1 & \cdots & \beta_{m+l-1} & \text{ } \end{array} \right).$$}
\normalsize{Occasionally} we think of $\left(\begin{array}{ccccc} \text{ }&\alpha_1  & \cdots & \alpha_{m+l-1}& \alpha_{m+l}\\ \beta_0 & \beta_1 & \cdots & \beta_{m+l-1} & \text{ } \end{array} \right)$ as 
$$\left(\begin{array}{ccccc} \text{ }&\alpha_1  & \cdots & \alpha_{m-1}& \alpha_m\\ \beta_0 & \beta_1  & \cdots & \beta_{m-1} & \beta_{m}\end{array} \right)\left(\begin{array}{cccc} \alpha_{m+1}  & \cdots & \alpha_{m+l-1}& \alpha_{m+l}\\  \beta_{m+1}  & \cdots & \beta_{m+l-1} & \text{ } \end{array} \right).$$

If we have a map $G : (x, y) \mapsto (xyx^{-1}, y^rx^{-1})$ and define $(z_n, w_n):=G^n(x, y)$, then it is easy to see that $z_n$ is one of the terms in $x_n$. Since we will express $x_n$ as the sum of monomials with the same number of columns as $z_n$ has, we want to find an expression for $z_n$, which requires the following definition. 

\begin{defn}
Let $\{c_n\}$ be the sequence  defined by the recurrence relation $$c_n=rc_{n-1} -c_{n-2},$$ with the initial condition $c_1=0$, $c_2=1$. Actually, when $r>2$, it is easy to see that 
$$
c_n= \frac{1}{\sqrt{r^2-4}  }\left(\frac{r+\sqrt{r^2-4}}{2}\right)^n - \frac{1}{\sqrt{r^2-4}  }\left(\frac{r-\sqrt{r^2-4}}{2}\right)^n = \sum_{i\geq 0} (-1)^i { {n-2-i} \choose i }r^{n-2-2i}.
$$ \qed
\end{defn}

Then we have
\begin{equation}\label{z_n}z_n=\left(\begin{array}{ccccccc} \text{ }&1 & -1 & -1 & \cdots & -1& -1\\ 0&1& b_{n,1} & b_{n, 2} & \cdots & b_{n, c_{n}} & \text{ }\end{array} \right)\end{equation}
for some integers $b_{n,1},\cdots, b_{n, c_{n}}$, which are defined as follows.

\begin{defn}
The sequence $\{b_{i,j}\}_{i,j\in \mathbb{Z}_{>0}}$ is recursively defined by:
$$b_{i,j}=0\text{ for }j>  c_i,$$
$$b_{2,1}=r-1, \text{ and}$$
$b_{n,1}=r-1,\, b_{n, p}=r \,\,\,\, (2+\sum_{i=1}^{j-1} b_{n-1,i}\leq p\leq\sum_{i=1}^j b_{n-1,i}),\, b_{n, 1+\sum_{i=1}^j b_{n-1,i}}=r-1$ for $n\geq 3$ and $1\leq j\leq c_{n-1}.$

It is easy to see that $1+\sum_{i=1}^{c_{n-1}} b_{n-1,i}=c_n$. The reason that (\ref{z_n}) holds is due to the following calculation:$$G(y^t x^{-1})=(y^rx^{-1})^{t}xy^{-1}x^{-1}=(y^rx^{-1})^{t-1}y^{r-1}x^{-1}\text{ for }t\geq 1.$$
\qed\end{defn}

\begin{exmp}
Let $r=3$. Then
$$
z_1=\left(\begin{array}{ccc} \text{ }&1 & -1 \\ 0&1& \text{ }\end{array} \right),
$$
$$z_2=\left(\begin{array}{cccc} \text{ }&1 & -1 &  -1\\ 0&1& 2 & \text{ }\end{array} \right),$$
$$z_3=\left(\begin{array}{cccccc} \text{ }&1 & -1 &  -1& -1 &  -1\\ 0&1& 2 & 3 & 2 & \text{ }\end{array} \right), \text{ and }$$
$$z_4=\left(\begin{array}{ccccccccccc} \text{ }&1 & -1 &  -1& -1 &  -1 & -1 &  -1& -1 &  -1& -1 \\ 0&1& 2 & 3 & 2 & 3 & 3 & 2 & 3 & 2 &\text{ }\end{array} \right).$$ \qed
\end{exmp}

We want to express $x_n$ as 
$$
x_n = \sum_{\alpha_i, \beta_i} \left(\begin{array}{ccccccc} \text{ }&1 & \alpha_1 & \alpha_2 & \cdots & \alpha_{c_n}& -1\\ 0&1& \beta_{1} & \beta_{2} & \cdots & \beta_{c_{n}} & \text{ }\end{array} \right),
$$
where the summation runs over $\alpha_i, \beta_i$ satisfying certain conditions. We believe that this expression is better than the minimal expression. We need to be able to locate zeroes in a systematic way. For this purpose, we will use the following notations.

\begin{defn}
Define a transformation $g$ on the set $\{(w_1,\cdots,w_j)\,| \, j\in \mathbb{Z}_{>0}, w_i=r\text{ or }r-1 \}$ of strings by

$\,\,\,\,\,\,\,\,\,\,\,\,\,\begin{array}{cc} g(r)= & \underbrace{(r,\cdots,r,r-1)} \\ \text{ } & r\text{ numbers}\end{array}$, \,\,\,\,\,\, $\begin{array}{cc} g(r-1)= & \underbrace{(r,\cdots,r,r-1),} \\ \text{ } & (r-1)\text{ numbers}\end{array}$,  \,\,\,\,\,\, and
$$g^s(w_1,\cdots,w_i, w_{i+1},\cdots,w_j)=g^s(w_1,\cdots,w_i)g^s(w_{i+1},\cdots,w_j)\text{ for any integers }i,j,s\geq 0,$$where $g^s=\begin{array}{c}\underbrace{g \circ g \circ \cdots \circ g}\\ s \end{array}$.
 
 Then we define the set of exceptional strings by $$\text{Exc}:=\{g^{s_1}(r)\cdots g^{s_j}(r) \,|\, j\in \mathbb{Z}_{>0}, \, s_1,...,s_j\in \mathbb{Z}_{\geq 0}  \}.$$ 
\end{defn}

\begin{rem}\label{Excrem}
For any $\mathbf{w}=(w_1,\cdots,w_j)\in \text{Exc}$ and any $i$ $(1\leq i\leq j)$, its substring $(w_1,\cdots,w_i)$ consisting of the first $i$ numbers in $\mathbf{w}$ also belongs to $\text{Exc}$.\end{rem}
\begin{exmp}
Let $r=3$. Then $g^0(3)=(3),\, g^0(3)g^0(3)=(3,3), \,g(3)=(3,3,2),\,g(3)g(3)=(3,3,2,3,3,2),\,\text{ and }g^2(3)=(3,3,2,3,3,2,3,2)$ belong to $\text{Exc}$. But $(2)$ and $(3,3,2,3,2)$ do not.
\qed\end{exmp}

Let $[m]$ denote the set $\{1,2,\cdots,m\}$ for any nonnegative integer $m$.

\begin{defn}\label{def_of_f}
We need a function $f$ from $\{\text{subsets of } [c_{n-1}]\}$ to $\{\text{subsets of } [c_{n}]\}$. For each subset $V\subset [c_{n-1}]$, we define $f(V)$ as follows.

If $V=\emptyset$ then $f(\emptyset)=\emptyset$. If $V\neq\emptyset$ then we write $V=\cup_{i=1}^j\{e_i,e_i+1,\cdots, e_i+l_i-1\}$ with $l_i>0$ $(1\leq i\leq j)$ and $e_i+l_i<e_{i+1}$ $(1\leq i\leq j-1)$. For each $1\leq i\leq j$, define $f_i(V)$ by
$$
f_i(V):=\left\{ \begin{array}{ll}  
[1+\sum_{k=1}^{e_i+l_i-1}b_{n-1,k}]\setminus [1+\sum_{k=1}^{e_i-1}b_{n-1,k}],  & \text{ if }(b_{n-1,e_i},\cdots,b_{n-1,e_i+l_i-1})\in \text{Exc}  \\
 \text{ } & \text{ }\\
{[}1+\sum_{k=1}^{e_i+l_i-1}b_{n-1,k}{]}\setminus [\sum_{k=1}^{e_i-1}b_{n-1,k}], & \text{ otherwise.} \end{array}  \right.
$$
Then $f(V)$ is obtained by taking the union of $f_i(V)$'s:
$$
f(V):=\cup_{i=1}^j f_i(V).
$$
Note that not every subset of $[c_n]$ can be realized as $f(V)$.
\qed\end{defn}

\begin{exmp}
Let $r=3$. If $n=3$ then $[c_2]=\{1\}$, and $f(\emptyset)=\emptyset\subset [c_3]$ and $f(\{1\})=\{1,2,3\}$. 

If $n=4$ then $[c_3]=\{1,2,3\}$. In this case, $f(\{2\})=\{4,5,6\}$ since $(b_{3,2})=(3)\in \text{Exc}$. On the other hand, $f(\{1\}\cup\{3\})=\{1,2,3\}\cup\{6,7,8\}\subset [c_4]$.
\qed\end{exmp}

\begin{defn}\label{restrict_to_term}
Let $$T=\sum_{\delta_1\in H_1}\cdots\sum_{\delta_{c_n}\in H_{c_n}} \left(\begin{array}{ccccccc} \text{ }&1 & \alpha_1& \alpha_2 & \cdots & \alpha_{c_n}& -1\\ 0&1& \beta_{1}-\delta_{1}r & \beta_{2}-\delta_{2}r & \cdots & \beta_{c_{n}}-\delta_{c_n}r & \text{ }\end{array} \right),$$ where each of $H_i$ is either $\{0\}$ or $\{0,1\}$. Let $V$ be any subset of $[c_n]$. Then we define
$$T|_V:=\left\{\begin{array}{ll} 0,  & \text{ if } H_i=\{0\} \text{ for some } i\in V \\
\text{ } & \text{ }\\
\begin{array}{l}\left(\begin{array}{ccccccc} \text{ }&1 & \alpha_1 & \alpha_2 & \cdots & \alpha_{c_n}& -1\\ 0&1& \beta_{1}-\delta_{1}r & \beta_{2}-\delta_{2}r & \cdots & \beta_{c_{n}}-\delta_{c_n}r & \text{ }\end{array} \right) \\\text{ where }\delta_{i}=1\text{ for }i\in V\text{ and }\delta_{i}=0\text{ for }i\not\in V, \end{array}
&  \text{ otherwise.}\end{array}\right.$$
Note that $T=\sum_{V\subset [c_n]} T|_V$.

In the proof of Lemma~\ref{mainlemma}, we will use the following minor modification. We let $T=T_1T_2$, where
$$T_1=\sum_{\delta_1\in H_1}\cdots\sum_{\delta_{m}\in H_{m}} \left(\begin{array}{ccccccc} \text{ }&1 & \alpha_1 & \alpha_2 & \cdots & \alpha_{m}&\alpha_{m+1}\\ 0&1& \beta_{1}-\delta_{1}r & \beta_{2}-\delta_{2}r & \cdots & \beta_{m}-\delta_{m}r \end{array} \right)$$and
$$T_2=\sum_{\delta_{m+1}\in H_{m+1}}\cdots\sum_{\delta_{c_n}\in H_{c_n}} \left(\begin{array}{ccccc} & \alpha_{m+2} & \cdots & \alpha_{c_n}& -1\\  \beta_{m+1}-\delta_{m+1}r & \beta_{m+2}-\delta_{m+2}r & \cdots & \beta_{c_{n}}-\delta_{c_n}r & \text{ }\end{array} \right).$$ Let $V$ be any subset of $[c_n]$ with $V=V_1\cup V_2$ $(V_1\subset [m], V_2 \subset [c_n]\setminus [m])$. By abuse of notation, we say that
$$
T|_V=T_1|_{V_1}T_2|_{V_2}.
$$
\qed\end{defn}

It is easy to see that
$$F(z_{n-2})=\sum_{\delta_1=0}^1\cdots\sum_{\delta_{c_{n-1}}=0}^1\left(\begin{array}{ccccccc} \text{ }&1 & -1 & -1 & \cdots & -1& -1\\ 0&1& b_{n-1,1}-\delta_1r & b_{n-1, 2}-\delta_2r & \cdots & b_{n-1, c_{n-1}}-\delta_{c_{n-1}}r & \text{ }\end{array} \right).$$ Some examples of $F(z_{n-2})|_V$ are given in Example~\ref{main_example1} and~\ref{main_example2}.

For monomials $A$ and $A_i$ in $K$, we let $\prod_{i=1}^m A_i$ denote $A_1A_2\cdots A_m$, and $A^m$ denote $\begin{array}{c}\underbrace{AA\cdots A}\\m  \end{array}$.

\begin{defn}\label{def_of_tildeF}
Let $$w={F}(z_{n-1})|_V= \left(\begin{array}{ccccccc} \text{ }&1 & -1 & -1 & \cdots & -1& -1\\ 0&1& b_{n,1}-\delta_{1}r & b_{n,2}-\delta_{2}r & \cdots & b_{n,c_{n}}-\delta_{c_n}r & \text{ }\end{array} \right),$$
where $\delta_{i}=1$ for $i\in V$ and $\delta_{i}=0$ for $i\not\in V$. We will define $\tilde{F}(w)$, which is a  modification of taking $F$, where we do not allow $(1+y^r)^{-1}$ to appear. In fact, if none of $b_{n,i}-\delta_{i}r$ is negative, then $\tilde{F}(w)=F(w)$. Remember that $b_{n,1}=r-1$ and that $b_{n,i}$ $(1<i\leq c_n)$ is either $r$ or $r-1$. In general, $\tilde{F}(w)$ can be defined by letting
$$\tilde{F}(w)=\tilde{F}\left(\begin{array}{ccc} \text{ }&1 & -1 \\ 0&1& r-1-\delta_{1}r \end{array} \right)\left[\prod_{i=2}^{c_n} \tilde{F}\left(\begin{array}{c} -1\\b_{n,i}-\delta_{i}r \end{array} \right)\right]\left(\begin{array}{c} -1\\ \text{ } \end{array} \right)$$
and using the following rules:
\tiny{$$\aligned &\tilde{F}\left(\begin{array}{ccc} \text{ }&1 & -1\\ 0&1 & r-1-\delta_1 r \end{array} \right)\\&:=\sum_{\delta^{(1)}=0}^{1-\delta_{1}}\cdots\sum_{\delta^{(r-1)}=0}^{1-\delta_{1}}\sum_{\delta^{(r)}=0}^{\Delta_{V,1}(1-\delta_{1})}\begin{array}{c}\underbrace{\left(\begin{array}{cccccccc} \text{ }&1& -1& \delta_1-1   & \cdots & \delta_1-1& 2\delta_1-1\\ 0& 1& (1-\delta^{(1)})r-1 & (1-\delta^{(2)})r&  \cdots & (1-\delta^{(r-1)})r & \Delta_{V,1}(1-\delta^{(r)})r-1 \end{array}\right)},\\  r+2 \text{ columns} \end{array}\endaligned$$}
$$
\tilde{F}\left(\begin{array}{c} -1\\b_{n,i}-\delta_{i}r \end{array} \right):=\sum_{\delta^{(1)}=0}^{1-\delta_{i}}\cdots\sum_{\delta^{(b_{n,i}-1)}=0}^{1-\delta_{i}}\sum_{\delta^{(b_{n,i})}=0}^{\Delta_{V,i}(1-\delta_{i})}\begin{array}{c}  \underbrace{\left( \begin{array}{cccc} \delta_i-1 &  \cdots & \delta_i-1 & (r-b_{n,i}+1)\delta_i-1 \\ (1-\delta^{(1)})r &\cdots& (1-\delta^{(b_{n,i}-1)})r & \Delta_{V,i}(1-\delta^{(b_{n,i})})r -1\end{array} \right)} \\ b_{n,i} \text{ columns }\end{array}  \text{ for }1<i\leq c_n,
$$
where 
$\tiny{\Delta_{V,i}=\left\{\begin{array}{ll} 0, &\text{ if }\delta_i=0, \delta_{i+1}=1,\text{ and }(b_{n,i+1}, \cdots, b_{n,i+e})\not\in \text{Exc}\text{ with }e
\text{ satisfying }\{i+1,\cdots, i+e\}\subset V\text{ and }i+e+1\not\in V\\ 
1, & \text{ otherwise}.    \end{array} \right.}$
\end{defn}

We are ready to state our main result.

\begin{thm}\label{mainthm}
Let $n\geq 3$. Let $z_{n-2}$ be the monomial corresponding to $$\left(\begin{array}{ccccccc} \text{ }& 1 & -1 & -1 & \cdots & -1& -1\\ 0&1& b_{n-2,1} & b_{n-2,2} & \cdots & b_{n-2,c_{n-2}} & \text{ }\end{array} \right).$$Then
\begin{equation}\label{maineq}x_n=\sum_{V\subset [c_{n-1}]} \left(\sum_{W\subset [c_n]\setminus f(V)} \tilde{F}({F}(z_{n-2})|_V)|_W\right)=\sum_{V\subset [c_{n-1}]} \left( \tilde{F}({F}(z_{n-2})|_V)\right). \end{equation}
\end{thm}

In addition to our theoretical proof, this formula is also checked by a computer up to the limit of practical computation ($r\leq 4$ and $n\leq 6$).

\begin{exmp}\label{main_example1}
Let $r=3$. Then $$F(z_1)|_{\emptyset}=\left(\begin{array}{cccc}\text{ }&  1 &  -1 & -1\\0& 1&2 & \text{ } \end{array} \right)\,\,\,\,\,\,\text{ and }\,\,\,\,\,\,\,F(z_1)|_{\{1\}}=\left(\begin{array}{cccc}\text{ }&  1 &  -1 & -1\\0& 1&-1 & \text{ } \end{array} \right),$$ which yield
$$\tilde{F}(F(z_1)|_{\emptyset})|_{\emptyset}=\left(\begin{array}{cccccc}\text{ }&  1 &  -1 & -1&-1&-1\\0& 1&2 &3&2& \text{ } \end{array} \right),\,\,\,\,\,\,\,\,\,\,\,\,\,\tilde{F}(F(z_1)|_{\emptyset})|_{\{1,2\}}=\left(\begin{array}{cccccc}\text{ }&  1 &  -1 & -1&-1&-1\\0& 1&-1 &0&2& \text{ } \end{array} \right),$$
$$\tilde{F}(F(z_1)|_{\emptyset})|_{\{1\}}=\left(\begin{array}{cccccc}\text{ }&  1 &  -1 & -1&-1&-1\\0& 1&-1 &3&2& \text{ } \end{array} \right),\,\,\,\,\,\,\,\,\,\,\,\,\,\tilde{F}(F(z_1)|_{\emptyset})|_{\{1,3\}}=\left(\begin{array}{cccccc}\text{ }&  1 &  -1 & -1&-1&-1\\0& 1&-1 &3&-1& \text{ } \end{array} \right),$$
$$\tilde{F}(F(z_1)|_{\emptyset})|_{\{2\}}=\left(\begin{array}{cccccc}\text{ }&  1 &  -1 & -1&-1&-1\\0& 1&2 &0&2& \text{ } \end{array} \right),\,\,\,\,\,\,\,\,\,\,\,\,\,\tilde{F}(F(z_1)|_{\emptyset})|_{\{2,3\}}=\left(\begin{array}{cccccc}\text{ }&  1 &  -1 & -1&-1&-1\\0& 1&2 &0&-1& \text{ } \end{array} \right),$$
$$\tilde{F}(F(z_1)|_{\emptyset})|_{\{3\}}=\left(\begin{array}{cccccc}\text{ }&  1 &  -1 & -1&-1&-1\\0& 1&2 &3&-1& \text{ } \end{array} \right),\,\,\,\,\,\,\,\,\,\,\,\,\,\tilde{F}(F(z_1)|_{\emptyset})|_{\{1,2,3\}}=\left(\begin{array}{cccccc}\text{ }&  1 &  -1 & -1&-1&-1\\0& 1&-1 &0&-1& \text{ } \end{array} \right),$$
$$\text{ and }\tilde{F}(F(z_1)|_{\{1\}})|_{\emptyset}=\left(\begin{array}{cccccc}\text{ }&  1 &  -1 &0&1&-1\\0& 1&-1 &0&-1& \text{ } \end{array} \right).$$
Hence $x_3$ has $9=\frac{2^3+1}{1}$ terms.
\end{exmp}

\begin{exmp}\label{main_example2}Let $r=3$.
$$\aligned F(z_2)|_{\emptyset}&=\left(\begin{array}{cccccc}\text{ }&  1 &  -1 & -1&-1&-1\\0& 1&2 &3&2& \text{ } \end{array} \right)\\
F(z_2)|_{\{1\}}&=\left(\begin{array}{cccccc}\text{ }&  1 &  -1 & -1&-1&-1\\0& 1&-1 &3&2& \text{ } \end{array} \right)\\
F(z_2)|_{\{2\}}&=\left(\begin{array}{cccccc}\text{ }&  1 &  -1 & -1&-1&-1\\0& 1&2 &0&2& \text{ } \end{array} \right)\\
F(z_2)|_{\{3\}}&=\left(\begin{array}{cccccc}&  1 &  -1 & -1&-1&-1\\0& 1&2 &3&-1& \end{array} \right)\endaligned\,\,\,\,\,\,\,\,\,\,\,\,\, \aligned
F(z_2)|_{\{1,2\}}&=\left(\begin{array}{cccccc}&  1 &  -1 & -1&-1&-1\\0& 1&-1 &0&2& \end{array} \right)\\
F(z_2)|_{\{1,3\}}&=\left(\begin{array}{cccccc}&  1 &  -1 & -1&-1&-1\\0& 1&-1 &3&-1& \end{array} \right)\\
F(z_2)|_{\{2,3\}}&=\left(\begin{array}{cccccc}&  1 &  -1 & -1&-1&-1\\0& 1&2 &0&-1& \end{array} \right)\\
F(z_2)|_{\{1,2,3\}}&=\left(\begin{array}{cccccc}&  1 &  -1 & -1&-1&-1\\0& 1&-1 &0&-1& \end{array} \right)
\endaligned $$
\tiny{
$$\aligned \tilde{F}(F(z_2)|_{\emptyset})=\sum_{\delta_1=0}^1\cdots\sum_{\delta_8=0}^1&\left(\begin{array}{cccccccccccc}\text{ }&  1 &  -1 & -1&-1& -1&-1&-1& -1&-1&-1\\0& 1&2-3\delta_1 &3-3\delta_2&2-3\delta_3& 3-3\delta_4&3-3\delta_5&2-3\delta_6&3-3\delta_7&2-3\delta_8& \text{ } \end{array} \right)\\
\tilde{F}(F(z_2)|_{\{1\}})=\sum_{\delta_4=0}^1\cdots\sum_{\delta_8=0}^1&\left(\begin{array}{cccccccccccc}\text{ }&  1 &  -1 & 0&1&-1& -1&-1&-1&-1& -1\\0& 1&2-3\cdot1 &3-3\cdot1&2-3\cdot1& 3-3\delta_4&3-3\delta_5&2-3\delta_6&3-3\delta_7&2-3\delta_8& \text{ } \end{array} \right)\\
\tilde{F}(F(z_2)|_{\{2\}})=\sum_{\delta_1=0}^1\cdots\sum_{\delta_8=0}^1&\left(\begin{array}{cccccccccccc}\text{ }&  1 &  -1 & -1&-1&0& 0&0&-1& -1&-1\\0& 1&2-3\delta_1 &3-3\delta_2&2-3\delta_3& 3-3\cdot1&3-3\cdot1&2-3\cdot1&3-3\delta_7&2-3\delta_8& \text{ }\end{array} \right)\\
 \tilde{F}(F(z_2)|_{\{3\}})=\sum_{\delta_1=0}^1\cdots\sum_{\delta_5=0}^1&\left(\begin{array}{cccccccccccc}\text{ }&  1 &  -1 & -1&-1& -1&-1&-1& 0&1&-1\\0& 1&2-3\delta_1 &3-3\delta_2&2-3\delta_3& 3-3\delta_4&3-3\delta_5&2-3\cdot1&3-3\cdot1&2-3\cdot1& \text{ } \end{array} \right)\\
\tilde{F}(F(z_2)|_{\{1,2\}})=\sum_{\delta_7=0}^1\sum_{\delta_8=0}^1&\left(\begin{array}{cccccccccccc}\text{ }&  1 &  -1 & 0&1&0& 0&0&-1& -1&-1\\0& 1&2-3\cdot1 &3-3\cdot1&2-3\cdot1& 3-3\cdot1&3-3\cdot1&2-3\cdot1&3-3\delta_7&2-3\delta_8& \text{ }\end{array} \right)\\
 \tilde{F}(F(z_2)|_{\{1,3\}})=\sum_{\delta_4=0}^1\sum_{\delta_5=0}^1&\left(\begin{array}{cccccccccccc}\text{ }&  1 &  -1 & 0&1& -1&-1&-1& 0&1&-1\\0& 1&2-3\cdot1 &3-3\cdot1&2-3\cdot1& 3-3\delta_4&3-3\delta_5&2-3\cdot1&3-3\cdot1&2-3\cdot1& \text{ } \end{array} \right)\\
 \tilde{F}(F(z_2)|_{\{2,3\}})=\sum_{\delta_1=0}^1\sum_{\delta_2=0}^1&\left(\begin{array}{cccccccccccc}\text{ }&  1 &  -1 & -1&-1& 0&0&0& 0&1&-1\\0& 1&2-3\delta_1 &3-3\delta_2&2-3\cdot1& 3-3\cdot1&3-3\cdot1&2-3\cdot1&3-3\cdot1&2-3\cdot1& \text{ } \end{array} \right)\\
 \tilde{F}(F(z_2)|_{\{1,2,3\}})=&\left(\begin{array}{cccccccccccc}\text{ }&  1 &  -1 & 0&1& 0&0&0& 0&1&-1\\0& 1&2-3\cdot1 &3-3\cdot1&2-3\cdot1& 3-3\cdot1&3-3\cdot1&2-3\cdot1&3-3\cdot1&2-3\cdot1& \text{ } \end{array} \right)
\endaligned$$
}
\normalsize{Hence $x_4$ has $(2^8+3\cdot 2^5+ 3\cdot 2^2+1\cdot 2^0)=\frac{(\frac{2^3+1}{1})^3+1}{2}$ terms.}
\end{exmp}

\section{Proofs}
\begin{lem}\label{mainlemma}
Let $W'$ be any subset of $[c_{n}]$.  Then
$$
{F}\left( \sum_{V: f(V)\subset W'} \tilde{F}(F(z_{n-2})|_V)|_{W'\setminus f(V)}  \right)
=\tilde{F}({F}(z_{n-1})|_{W'}).
$$
\end{lem}
\begin{lem}
Lemma~\ref{mainlemma} implies Theorem~\ref{mainthm}.
\end{lem}
\begin{proof}
The second equality in $(\ref{maineq})$ follows from Definition~\ref{restrict_to_term}. By induction on $n$, we assume that Theorem~\ref{mainthm} holds for $n$. Then
$$\aligned x_{n+1}
&=F(x_{n})\\
&= {F}\left( \sum_{V\subset [c_{n-1}]} \left(\sum_{W\subset [c_n]\setminus f(V)} \tilde{F}({F}(z_{n-2})|_V)|_W\right)  \right)\\
&=  {F}\left( \sum_{W\subset [c_n]} \left(\sum_{V: W\cap f(V)=\emptyset} \tilde{F}({F}(z_{n-2})|_V)|_W\right)  \right)\\
&\overset{(*)}={F}\left( \sum_{W'\subset [c_n]} \left(\sum_{V:f(V)\subset W'} \tilde{F}({F}(z_{n-2})|_V)|_{W'\setminus f(V)}\right)  \right)\\
&=  \sum_{W'\subset [c_n]} {F}\left(\sum_{V:f(V)\subset W'} \tilde{F}({F}(z_{n-2})|_V)|_{W'\setminus f(V)}\right) \\
&=  \sum_{W'\subset [c_n]}\tilde{F}({F}(z_{n-1})|_{W'}), \endaligned$$where the last equality follows from Lemma~\ref{mainlemma}.

To show $(*)$, we give a bijection from
$$\{(V,W) \,|\, W\subset [c_n], W\cap f(V)=\emptyset \}$$
to
$$
\{(V,W') \,|\, W'\subset [c_n], V:f(V)\subset W' \},
$$
which is defined by $(V,W)\mapsto (V,W\cup f(V))$ and its inverse $(V,W')\mapsto (V, W'\setminus f(V))$.
\end{proof}

\begin{proof}[Sketch of proof of Lemma~\ref{mainlemma}]
If $W'=\emptyset$ then the statement is an easy consequence of Definition~\ref{def_of_tildeF} together with $F(z_{n-2})|_{\emptyset}=z_{n-1}$. If $W'\neq\emptyset$, then there exist integers $j,e_i, l_i>0$ such that $W'=\cup_{i=1}^j\{e_i,e_i+1,\cdots, e_i+l_i-1\}$ with $e_i+l_i<e_{i+1}$ for all $1\leq i\leq j-1$. We use induction on $j$. We prove the base case of $j=1$ in the next Lemma, which has an essential idea. In fact, once we establish the base case, then the case of $j>1$ is straightforward. Here we will sketch how the former case implies the latter case.

Suppose that $j>1$. Then there is an integer $m'$ such that $\min W'< m'< \max W'$ but $m'+1\not\in W'$.
Let $W'_1=\{i\in W' \,|\, i<m'+1\}$ and $W'_2=\{i\in W' \,|\, i>m'+1\}$. Let $m$ be the integer satisfying $f(f([m]))\subset [m']$ but $f(f([m+1]))\not\subset [m']$.
Let $$(z_{n-2})_1=\left(\begin{array}{ccccccc}& 1 & -1 & -1 & \cdots & -1&-1\\ 0&1& b_{n-2,1} & b_{n-2,2} & \cdots & b_{n-2,m}&\end{array} \right),$$ and $$(z_{n-2})_2=\left(\begin{array}{ccccc}  & -1 & \cdots & -1& -1\\ b_{n-2,m+1}& b_{n-2,m+2} &  \cdots & b_{n-2,c_{n-2}} & \end{array} \right).$$
Let $$(z_{n-1})_1=\left(\begin{array}{ccccccc}& 1 & -1 & -1 & \cdots & -1&-1\\ 0&1& b_{n-1,1} & b_{n-1,2} & \cdots & b_{n-1,1+\sum_{i=1}^m b_{n-2,i}}&\end{array} \right),$$ and $$(z_{n-1})_2=\left(\begin{array}{ccccc}  & -1 & \cdots & -1& -1\\ b_{n-1,2+\sum_{i=1}^m b_{n-2,i}}& b_{n-1,3+\sum_{i=1}^m b_{n-2,i}} &  \cdots & b_{n-1,c_{n-1}} & \end{array} \right).$$
Then
$$\aligned
&{F}\left( \sum_{V: f(V)\subset W'} \tilde{F}(F(z_{n-2})|_V)|_{W'\setminus f(V)}  \right)\\
&={F}\left( \sum_{V: f(V)\subset W'_1\cup W'_2} \tilde{F}(F(z_{n-2})|_V)|_{ (W'_i\cup W'_2)\setminus f(V)}  \right)\\
&={F}\left( \sum_{V_1: f(V_1)\subset W'_1} \,\,\,\,\,\sum_{V_2: f(V_2)\subset W'_2}\tilde{F}(F((z_{n-2})_1 (z_{n-2})_2)|_{V_1\cup V_2})|_{(W'_1\setminus f(V_1))\cup (W'_2\setminus f(V_2))}  \right)\\
&={F}\left( \sum_{V_1: f(V_1)\subset W'_1} \tilde{F}(F((z_{n-2})_1)|_{V_1})|_{W'_1\setminus f(V_1)}  \right){F}\left( \sum_{V_2: f(V_2)\subset W'_2} \tilde{F}(F((z_{n-2})_2)|_{V_2})|_{W'_2\setminus f(V_2)}  \right)\\
&=\tilde{F}({F}((z_{n-1})_1)|_{W'_1})\tilde{F}({F}((z_{n-1})_2)|_{W'_2})\,\,\,\,\,\,\,\,\,\,(\text{this is implied by induction on }j \text{ and abuse of }\tilde{F})\\
&=\tilde{F}({F}(z_{n-1})|_{W'}).\\
\endaligned$$
\end{proof}

\begin{lem}\label{sublemma}
Let $W'$ be a subset of $[c_{n}]$, which consists of consecutive numbers.  Then
$$
{F}\left( \sum_{V: f(V)\subset W'} \tilde{F}(F(z_{n-2})|_V)|_{W'\setminus f(V)}  \right)
=\tilde{F}({F}(z_{n-1})|_{W'}).
$$
\end{lem}

\begin{proof}
The reader is recommended to see Example~\ref{exmp3}. 

Let $V^M=\cup_{f(V)\subset W'} V$. Since $W'\subset [c_{n}]$ is a set of consecutive numbers, Definition~\ref{def_of_f} implies that $V^M$ also consists of consecutive numbers. Write $$V^M= \{e,e+1,\cdots,e+l-1\}.$$ Assume that $e>1$ (The case of $e=1$ can be done in a similar manner). For each $V\subset V^M$, we have
$$\aligned F(z_{n-2})|_V=&\left(\begin{array}{ccc} \text{ }&1 & -1 \\ 0&1& r-1 \end{array} \right)\left[\prod_{i=2}^{e-1} \left(\begin{array}{c} -1\\b_{n-1,i} \end{array} \right)\right]\\
&\times \left[\prod_{i=e}^{e+l-1} \left(\begin{array}{c} -1\\b_{n-1,i}-\delta_{i}(V)r \end{array} \right)\right]\left[\prod_{i=e+l}^{c_{n-1}} \left(\begin{array}{c} -1\\b_{n-1,i} \end{array} \right)\right]\left(\begin{array}{c} -1\\ \text{ } \end{array} \right),\endaligned$$
where $\delta_{i}(V)=1$ for $i\in V$ and $\delta_{i}(V)=0$ for $i\not\in V$. 

On the other hand, there are non-negative integers $s_1$ and $s_2$ such that $$W'=[s_2+1+\sum_{k=1}^{e+l-1}b_k]\setminus   [-s_1+1+\sum_{k=1}^{e-1}b_k].$$  Note that   we have $$1+\sum_{k=1}^{e-2}b_k< -s_1+2+\sum_{k=1}^{e-1}b_k$$ and  $$s_2+1+\sum_{k=1}^{e+l-1}b_k< 1+\sum_{k=1}^{e+l}b_k,$$ because otherwise $\cup_{f(V)\subset W'} V\neq \{e,e+1,\cdots,e+l-1\}$. It is worth mentioning that $s_1=0$ implies $(b_{n-1,e},b_{n-1,e+1},\cdots,b_{n-1,e+l-1})\in \text{Exc}$.

If $s_1>0$ then the desired statement is obtained by the following computations:
$$\aligned
&{F}\left( \sum_{V: f(V)\subset W'} \tilde{F}(F(z_{n-2})|_V)|_{W'\setminus f(V)}  \right)\\
&={F}\left( \sum_{V: f(V)\subset W'} \tilde{F}\left(\left(\begin{array}{ccc} \text{ }&1 & -1 \\ 0&1& r-1\end{array} \right)\left[\prod_{i=2}^{e-1} \left(\begin{array}{c} -1\\b_{n-1,i} \end{array} \right)\right] \right.\right.\\
&\,\,\,\,\,\,\,\,\,\,\,\times \left.\left.\left[\prod_{i=e}^{e+l-1} \left(\begin{array}{c} -1\\b_{n-1,i}-\delta_{i}(V)r \end{array} \right)\right]\left[\prod_{i=e+l}^{c_{n-1}} \left(\begin{array}{c} -1\\b_{n-1,i} \end{array} \right)\right]\left(\begin{array}{c} -1\\ \text{ } \end{array} \right) \right)|_{W'\setminus f(V)}  \right)\,\,\,\,\,\,\,\,\,\,\,\,\,\,\,\,\,\,\,\,\,\,\,\,\,\,\,\,\,\,\,\,\,\,\,\endaligned$$
$$\aligned
&={F}\left( \sum_{V: f(V)\subset W'} \left(\sum_{j=1}^{r}\sum_{\delta^{(j)}=0}^{1}\begin{array}{c}\underbrace{\left(\begin{array}{cccccccc} \text{ }&1& -1&-1   & \cdots & -1& -1\\ 0& 1& (1-\delta^{(1)})r-1 & (1-\delta^{(2)})r&  \cdots & (1-\delta^{(r-1)})r & (1-\delta^{(r)})r-1 \end{array}\right)}\\  r+2 \text{ columns} \end{array}\right.\right.\\
&\,\,\,\,\,\,\,\,\,\,\,\times\left[\prod_{i=2}^{e-2}\sum_{j=1}^{b_{n-1,i}}\sum_{\delta^{(j)}=0}^{1}\begin{array}{c}  \underbrace{\left( \begin{array}{cccc} -1 &  \cdots & -1 & -1 \\ (1-\delta^{(1)})r &\cdots& (1-\delta^{(b_{n-1,i}-1)})r & (1-\delta^{(b_{n-1,i})})r -1\end{array} \right)} \\ b_{n-1,i} \text{ columns }\end{array}\right] \\
&\,\,\,\,\,\,\,\,\,\,\,\times\sum_{j=1}^{b_{n-1,e-1}-1}\sum_{\delta^{(j)}=0}^{1}\sum_{\delta^{(b_{n-1,e-1})}=0}^{\Delta_{V,e-1}}\begin{array}{c}  \underbrace{\left( \begin{array}{cccc} -1 &  \cdots & -1 & -1 \\ (1-\delta^{(1)})r &\cdots& (1-\delta^{(b_{n-1,e-1}-1)})r &   \Delta_{V,e-1}(1-\delta^{(b_{n-1,e-1})})r-1\end{array} \right)} \\ b_{n-1,e-1} \text{ columns }\end{array} \\
&\,\,\,\,\,\,\,\,\,\,\,\times\left[\prod_{i=e}^{e+l-1}  \sum_{j=1}^{b_{n-1,i}-1}\sum_{\delta^{(j)}=0}^{1-\delta_{i}}\sum_{\delta^{(b_{n-1,i})}=0}^{\Delta_{V,i}(1-\delta_{i})}\begin{array}{c}  \underbrace{\left( \begin{array}{cccc} \delta_i(V)-1 &  \cdots & \delta_i(V)-1 & (r-b_{n-1,i}+1)\delta_i(V)-1 \\ (1-\delta^{(1)})r &\cdots& (1-\delta^{(b_{n-1,i}-1)})r & \Delta_{V,i}(1-\delta^{(b_{n-1,i})})r -1\end{array} \right)} \\ b_{n-1,i} \text{ columns }\end{array}\right]\\
&\,\,\,\,\,\,\times \left.\left.\left[\prod_{i=e+l}^{c_{n-1}} \sum_{j=1}^{b_{n-1,i}}\sum_{\delta^{(j)}=0}^{1}\begin{array}{c}  \underbrace{\left( \begin{array}{cccc} -1 &  \cdots & -1 & -1 \\ (1-\delta^{(1)})r &\cdots& (1-\delta^{(b_{n-1,i}-1)})r & (1-\delta^{(b_{n-1,i})})r -1\end{array} \right)} \\ b_{n-1,i} \text{ columns }\end{array}\right]\left(\begin{array}{c} -1\\ \text{ } \end{array} \right) \right)|_{W'\setminus f(V)}  \right)
\endaligned$$\footnote{By abuse of notation, $\sum_{j=1}^{b_{n-1,i}-1}\sum_{\delta^{(j)}=0}^{1-\delta_{i}}$ means $\sum_{\delta^{(1)}=0}^{1-\delta_{i}}\cdots\sum_{\delta^{(b_{n-1,i}-1)}=0}^{1-\delta_{i}}$. These are made due to lack of space.}
$$\aligned
&={F}\left( \sum_{V: V\subset V^M} {\left(\begin{array}{ccc} \text{ }&1& -1\\ 0& 1& r-1  \end{array}\right)\left(\begin{array}{c}  -1\\ r  \end{array}\right)^{r-2}\left(\begin{array}{c}-1  \\ r-1 \end{array}\right)                }\prod_{i=2}^{e-2}\left[{\left( \begin{array}{c} -1  \\ r \end{array} \right)^{b_{n-1,i}-1}               \left( \begin{array}{c}  -1 \\  r -1\end{array} \right)}\right] \right.\\
&\,\,\,\,\,\,\,\,\,\,\,\times\left(\begin{array}{c} -1  \\ r \end{array}\right)^{b_{n-1,e-1}-s_1}           \begin{array}{c}  \underbrace{\left( \begin{array}{cccc} -1 &  \cdots & -1 & -1 \\ 0 &\cdots& 0 &   -1\end{array} \right)}\\ s_1\text{ columns}\end{array}  \\
&\,\,\,\,\,\,\,\,\,\,\,\times\prod_{i=e}^{e+l-1}\left[ {\left( \begin{array}{c} \delta_i(V)-1  \\ 0 \end{array} \right)^{b_{n-1,i}-1}                                            \left( \begin{array}{c}  (r-b_{n-1,i}+1)\delta_i(V)-1 \\ -1\end{array} \right)} \right]\\
&\,\,\,\,\,\,\,\,\,\,\,\times \left. \left( \begin{array}{c} -1 \\ 0\end{array} \right)^{s_2}\left( \begin{array}{c} -1 \\ r\end{array} \right)^{b_{n-1,e+l}-s_2-1}         \left( \begin{array}{c} -1\\ r -1\end{array} \right)\prod_{i=e+l+1}^{c_{n-1}}\left[  {\left( \begin{array}{c} -1 \\ r\end{array} \right)^{b_{n-1,i}-1}          \left( \begin{array}{c} -1  \\r -1\end{array} \right)} \right]\left(\begin{array}{c} -1\\ \text{ } \end{array} \right)  \right)
\endaligned$$
$$\aligned
&={F}\left(  {\left(\begin{array}{ccc} \text{ }&1& -1\\ 0& 1& r-1  \end{array}\right)\left(\begin{array}{c}  -1\\ r  \end{array}\right)^{r-2}\left(\begin{array}{c}-1  \\ r-1 \end{array}\right)                }\prod_{i=2}^{e-2}\left[{\left( \begin{array}{c} -1  \\ r \end{array} \right)^{b_{n-1,i}-1}               \left( \begin{array}{c}  -1 \\  r -1\end{array} \right)}\right] \right.\\
&\,\,\,\,\,\,\,\,\,\,\,\times\left(\begin{array}{c} -1  \\ r \end{array}\right)^{b_{n-1,e-1}-s_1}           \begin{array}{c}  \underbrace{\left( \begin{array}{cccc} -1 &  \cdots & -1 & -1 \\ 0 &\cdots& 0 &   -1\end{array} \right)}\\ s_1\text{ columns}\end{array}  \\
&\,\,\,\,\,\,\,\,\,\,\,\times\prod_{i=e}^{e+l-1} \sum_{\delta_i(V)=0}^{1}\left[ {\left( \begin{array}{c} \delta_i(V)-1  \\ 0 \end{array} \right)^{b_{n-1,i}-1}                                            \left( \begin{array}{c}  (r-b_{n-1,i}+1)\delta_i(V)-1 \\ -1\end{array} \right)} \right]\\
&\,\,\,\,\,\,\,\,\,\,\,\times \left. \left( \begin{array}{c} -1 \\ 0\end{array} \right)^{s_2}\left( \begin{array}{c} -1 \\ r\end{array} \right)^{b_{n-1,e+l}-s_2-1}         \left( \begin{array}{c} -1\\ r -1\end{array} \right)\prod_{i=e+l+1}^{c_{n-1}}\left[  {\left( \begin{array}{c} -1 \\ r\end{array} \right)^{b_{n-1,i}-1}          \left( \begin{array}{c} -1  \\r -1\end{array} \right)} \right]\left(\begin{array}{c} -1\\ \text{ } \end{array} \right)  \right)
\endaligned$$
$$\aligned
&= {{F} \left[\left(\begin{array}{ccc} \text{ }&1& -1\\ 0& 1& r-1  \end{array}\right)\left(\begin{array}{c}  -1\\ r  \end{array}\right)^{r-2}\left(\begin{array}{c}-1  \\ r-1 \end{array}\right)          \right]      }\prod_{i=2}^{e-2}{F} \left[{\left( \begin{array}{c} -1  \\ r \end{array} \right)^{b_{n-1,i}-1}               \left( \begin{array}{c}  -1 \\  r -1\end{array} \right)}\right]\\
&\,\,\,\,\,\,\,\,\,\,\,\times{F}\left[ \left(\begin{array}{c} -1  \\ r \end{array}\right)^{b_{n-1,e-1}-s_1}           \begin{array}{c}  \underbrace{\left( \begin{array}{cccc} -1 &  \cdots & -1 & -1 \\ 0 &\cdots& 0 &   -1\end{array} \right)}\\ s_1\text{ columns}\end{array}\right]  \\
&\,\,\,\,\,\,\,\,\,\,\,\times\prod_{i=e}^{e+l-1} \sum_{\delta_i(V)=0}^{1}{F} \left[ {\left( \begin{array}{c} \delta_i(V)-1  \\ 0 \end{array} \right)^{b_{n-1,i}-1}                                            \left( \begin{array}{c}  (r-b_{n-1,i}+1)\delta_i(V)-1 \\ -1\end{array} \right)} \right]\\
&\,\,\,\,\,\,\,\,\,\,\,\times {F}\left[ \left( \begin{array}{c} -1 \\ 0\end{array} \right)^{s_2}\left( \begin{array}{c} -1 \\ r\end{array} \right)^{b_{n-1,e+l}-s_2-1}         \left( \begin{array}{c} -1\\ r -1\end{array} \right)\right]\prod_{i=e+l+1}^{c_{n-1}}{F} \left[  {\left( \begin{array}{c} -1 \\ r\end{array} \right)^{b_{n-1,i}-1}          \left( \begin{array}{c} -1  \\r -1\end{array} \right)} \right]F\left(\begin{array}{c} -1\\ \text{ } \end{array} \right)
\endaligned$$ $(**)$\footnote{The equality under $(**)$ is a key step, where cancellation with $(1+y^r)^{-1}$ occurs.}
$$\aligned
&={{F} \left[\left(\begin{array}{ccc} \text{ }&1& -1\\ 0& 1& r-1  \end{array}\right)\left(\begin{array}{c}  -1\\ r  \end{array}\right)^{r-2}\left(\begin{array}{c}-1  \\ r-1 \end{array}\right)   \right]             }\prod_{i=2}^{e-2}{F} \left[{\left( \begin{array}{c} -1  \\ r \end{array} \right)^{b_{n-1,i}-1}               \left( \begin{array}{c}  -1 \\  r -1\end{array} \right)}\right]\\
&\,\,\,\,\,\,\,\,\,\,\,\times{F} \left[\left(\begin{array}{c} -1  \\ r \end{array}\right)^{b_{n-1,e-1}-s_1}           \begin{array}{c}  \underbrace{\left( \begin{array}{cccc} -1 &  \cdots & -1 & -1 \\ 0 &\cdots& 0 &   -1\end{array} \right)}\\ s_1\text{ columns}\end{array} \right] \\
&\,\,\,\,\,\,\,\,\,\,\,\times\prod_{i=e}^{e+l-1}  \left(\begin{array}{c} 1   \\ -b_{n-1,i}\end{array}\right) \\
&\,\,\,\,\,\,\,\,\,\,\,\times {F}\left[ \left( \begin{array}{c} -1 \\ 0\end{array} \right)^{s_2}\left( \begin{array}{c} -1 \\ r\end{array} \right)^{b_{n-1,e+l}-s_2-1}         \left( \begin{array}{c} -1\\ r -1\end{array} \right) \right]  \prod_{i=e+l+1}^{c_{n-1}}{F} \left[  {\left( \begin{array}{c} -1 \\ r\end{array} \right)^{b_{n-1,i}-1}        \left( \begin{array}{c} -1  \\r -1\end{array} \right)} \right]F\left(\begin{array}{c} -1\\ \text{ } \end{array} \right)
\endaligned$$
$$\aligned
&= {{F}\left[ \left(\begin{array}{ccc} \text{ }&1& -1\\ 0& 1& r-1  \end{array}\right)\left(\begin{array}{c}  -1\\ r  \end{array}\right)^{r-2}\left(\begin{array}{c}-1  \\ r-1 \end{array}\right)         \right]       }\prod_{i=2}^{e-2}{F} \left[{\left( \begin{array}{c} -1  \\ r \end{array} \right)^{b_{n-1,i}-1}               \left( \begin{array}{c}  -1 \\  r -1\end{array} \right)}\right]\\
&\,\,\,\,\,\,\,\,\,\,\,\times{F} \left[\left(\begin{array}{c} -1  \\ r \end{array}\right)^{b_{n-1,e-1}-s_1}           \begin{array}{c}  \underbrace{\left( \begin{array}{cccc} -1 &  \cdots & -1 & -1 \\ 0 &\cdots& 0 &   -1\end{array} \right)}\\ s_1\text{ columns}\end{array} \right] \\
&\,\,\,\,\,\,\,\,\,\,\,\times\prod_{i=e}^{e+l-1} \left[ \left( \begin{array}{c}0 \\ 0 \end{array} \right)^{r-2} \left( \begin{array}{c}1 \\ -1 \end{array} \right)    
\left(\left( \begin{array}{c}0 \\ 0 \end{array} \right)^{r-1}\left( \begin{array}{c}0 \\ -1 \end{array} \right)  \right)^{b_{n-1,i}-1}\right]  \\
&\,\,\,\,\,\,\,\,\,\,\,\times {F}\left[ \left( \begin{array}{c} -1 \\ 0\end{array} \right)^{s_2}\left( \begin{array}{c} -1 \\ r\end{array} \right)^{b_{n-1,e+l}-s_2-1}         \left( \begin{array}{c} -1\\ r -1\end{array} \right)\right]     \prod_{i=e+l+1}^{c_{n-1}}{F} \left[  {\left( \begin{array}{c} -1 \\ r\end{array} \right)^{b_{n-1,i}-1}          \left( \begin{array}{c} -1  \\r -1\end{array} \right)} \right]F\left(\begin{array}{c} -1\\ \text{ } \end{array} \right)
\endaligned$$
$$\aligned
&= {{F}\left[ \left(\begin{array}{ccc} \text{ }&1& -1\\ 0& 1& r-1  \end{array}\right)\left(\begin{array}{c}  -1\\ r  \end{array}\right)^{r-2}\left(\begin{array}{c}-1  \\ r-1 \end{array}\right) \right]               }\prod_{i=2}^{e-2}{F} \left[{\left( \begin{array}{c} -1  \\ r \end{array} \right)^{b_{n-1,i}-1}               \left( \begin{array}{c}  -1 \\  r -1\end{array} \right)}\right]\\
&\,\,\,\,\,\,\,\,\,\,\,\times{F} \left[\left(\begin{array}{c} -1  \\ r \end{array}\right)^{b_{n-1,e-1}-s_1}           \begin{array}{c}  \underbrace{\left( \begin{array}{cccc} -1 &  \cdots & -1 & -1 \\ 0 &\cdots& 0 &   -1\end{array} \right)}\\ s_1\text{ columns}\end{array} \right] \\
&\,\,\,\,\,\,\,\,\,\,\,\times\prod_{i=e}^{e+l-1}\left[ \tilde{F}\left( \begin{array}{c}-1 \\ -1 \end{array} \right)    
\left(\tilde{F}\left( \begin{array}{c}-1 \\ 0 \end{array} \right)  \right)^{b_{n-1,i}-1} \right] \\
&\,\,\,\,\,\,\,\,\,\,\,\times {F}\left[ \left( \begin{array}{c} -1 \\ 0\end{array} \right)^{s_2}\left( \begin{array}{c} -1 \\ r\end{array} \right)^{b_{n-1,e+l}-s_2-1}         \left( \begin{array}{c} -1\\ r -1\end{array} \right)\right]  \prod_{i=e+l+1}^{c_{n-1}}{F} \left[  {\left( \begin{array}{c} -1 \\ r\end{array} \right)^{b_{n-1,i}-1}          \left( \begin{array}{c} -1  \\r -1\end{array} \right)} \right]F\left(\begin{array}{c} -1\\ \text{ } \end{array} \right) \\
&=\tilde{F}({F}(z_{n-1})|_{W'}).
\endaligned$$

Even if $s_1=0$, we use the same argument with $\left(\begin{array}{c} -1  \\ r \end{array}\right)^{b_{n-1,e-1}-s_1}$ being replaced by  $\left(\begin{array}{c} -1  \\ r \end{array}\right)^{b_{n-1,e-1}-1}\left(\begin{array}{c} -1  \\ r-1 \end{array}\right)$. Remark~\ref{Excrem} explains why the same proof works for the case of $s_1=0$. 
\end{proof}

\begin{exmp}\label{exmp3}Let $r=3$.
 Then $z_2=\left(\begin{array}{cccc}  &1 &  -1 & -1\\ 0&1&2 &  \end{array} \right)$. 
So we have
$$F(z_2)|_{\emptyset}=\left(\begin{array}{cccccc}&  1 &  -1 & -1&  -1& -1\\ 0&1&2 & 3 &2 & \end{array} \right)\text{ and }F(z_2)|_{\{1\}}=\left(\begin{array}{cccccc}  &1 &  -1 & -1&  -1& -1\\ 0&1&-1 & 3 &2 & \end{array} \right).$$
$$\aligned& \tiny{\tilde{F}(F(z_2)|_{\emptyset})=\sum_{\delta_1=0}^1\cdots\sum_{\delta_8=0}^1\left(\begin{array}{ccccccccccc} & 1 &  -1 & -1& -1& -1& -1& -1& -1& -1&-1\\ 0&1&2-3\delta_1 & 3-3\delta_2 &2-3\delta_3 & 3-3\delta_4&3-3\delta_5 & 2-3\delta_6&3-3\delta_7 & 2-3\delta_8&\end{array} \right).}
\endaligned$$
$$\aligned&\tiny{ \tilde{F}(F(z_2)|_{\{1\}})=\sum_{\delta_4=0}^1\cdots\sum_{\delta_8=0}^1\left(\begin{array}{cccccccccccc} & 1 &  -1 &0&1& -1& -1& -1& -1& -1& -1\\ 0&1&-1 &0&-1 & 3-3\delta_4&3-3\delta_5 & 2-3\delta_6&3-3\delta_7 & 2-3\delta_8\end{array} \right).}
\endaligned$$
$$\aligned& \tilde{F}(F(z_2)|_{\emptyset})|_{\{1,2,3\}}=\left(\begin{array}{ccccccccccc} & 1 &  -1 & -1& -1& -1& -1& -1& -1& -1&-1\\0& 1&-1 & 0 &-1 & 3&3 & 2&3 & 2&\end{array} \right).
\endaligned$$
$$\aligned& \tilde{F}(F(z_2)|_{\{1\}})|_{\emptyset}=\left(\begin{array}{cccccccccccc}  &1 &  -1 &0&1& -1& -1& -1& -1& -1& -1\\ 0&1&-1 &0&-1 & 3&3 & 2&3& 2\end{array} \right).
\endaligned$$
Then $$\aligned &F\left(\tilde{F}(F(z_2)|_{\emptyset})|_{\{1,2,3\}\setminus f(\emptyset)} + \tilde{F}(F(z_2)|_{\{1\}})|_{\{1,2,3\}\setminus f(\{1\})}  \right)|_{\emptyset}\\
&=F\left(\tilde{F}(F(z_2)|_{\emptyset})|_{\{1,2,3\}} + \tilde{F}(F(z_2)|_{\{1\}})|_{\emptyset}  \right)|_{\emptyset}\\
&=\tiny{\left(\begin{array}{cccccccccc} &1 &  -1&0& 1 &0&0 &0&0 &1   \\ 0&1&-1&0 &-1 &0&0&-1 &0 &-1  \end{array} \right)\left(\begin{array}{ccc} -1 &-1 &-1  \\ 3 &3 &2 \end{array} \right)^2\left(\begin{array}{cccccccc} -1 &-1&-1&-1&-1&-1&-1&-1   \\ 3 &2&3 &3 &2 &3 &2 &\end{array} \right)}\\
&=\tilde{F}(F(z_3)|_{\{1,2,3\}})|_{\emptyset}
\endaligned$$
\end{exmp}

\section{Rank 2 Cluster Algebra of affine type}\label{affine}
In this section, we would like to compare our formula with the known formula by Caldero and Zelevinsky \cite{CZ} among others \cite{SZ}, \cite{MP}. Fix $r=2$. 

\begin{defn}
Let $\{c_n\}$ be the sequence  defined by the recurrence relation $$c_n=2c_{n-1} -c_{n-2},$$ with the initial condition $c_1=0$, $c_2=1$. Then $c_n= n-1.$ \qed
\end{defn}

If we have a map $G : (x, y) \mapsto (xyx^{-1}, y^2x^{-1})$ and define $(z_n, w_n):=G^n(x, y)$, then it is easy to see that $z_n$ is one of the terms in $x_n$. Then we have
$$z_n=\begin{array}{c}\underbrace{\left(\begin{array}{ccccccc} \text{ }&1 & -1 & -1 & \cdots & -1& -1\\ 0&1& 1 & 1 & \cdots & 1 & \text{ }\end{array} \right)}\\ n+2\text{ columns}\end{array}.$$

When $r=2$, no exceptional string occurs, so we do not need $\text{Exc}$. When $r=2$, definition of $f$ can be reduced as follows.

\begin{defn}\label{def_of_f2}
We need a function $f$ from $\{\text{subsets of } [c_{n-1}]\}$ to $\{\text{subsets of } [c_{n}]\}$. For each subset $V\subset [c_{n-1}]$, we define $f(V)$ as follows.

If $V=\emptyset$ then $f(\emptyset)=\emptyset$. If $V\neq\emptyset$ then we write $V=\cup_{i=1}^j\{e_i,e_i+1,\cdots, e_i+l_i-1\}$ with $l_i>0$ $(1\leq i\leq j)$ and $e_i+l_i<e_{i+1}$ $(1\leq i\leq j-1)$. Then
$$
f(V):=\cup_{i=1}^j\{e_i,e_i+1,\cdots, e_i+l_i\}
$$
\qed\end{defn}

\begin{defn}\label{restrict_to_term2}
Let $$T=\sum_{\delta_1\in H_1}\cdots\sum_{\delta_{c_n}\in H_{c_n}} \left(\begin{array}{ccccccc} \text{ }&1 & \alpha_1& \alpha_2 & \cdots & \alpha_{c_n}& -1\\ 0&1& \beta_{1}-2\delta_{1} & \beta_{2}-2\delta_{2} & \cdots & \beta_{c_{n}}-2\delta_{c_n} & \text{ }\end{array} \right),$$ where each of $H_i$ is either $\{0\}$ or $\{0,1\}$. Let $V$ be any subset of $[c_n]$. Then we define
$$T|_V:=\left\{\begin{array}{ll} 0,  & \text{ if } H_i=\{0\} \text{ for some } i\in V \\
\text{ } & \text{ }\\
\begin{array}{l}\left(\begin{array}{ccccccc} \text{ }&1 & \alpha_1 & \alpha_2 & \cdots & \alpha_{c_n}& -1\\ 0&1& \beta_{1}-2\delta_{1} & \beta_{2}-2\delta_{2} & \cdots & \beta_{c_{n}}-2\delta_{c_n} & \text{ }\end{array} \right) \\\text{ where }\delta_{i}=1\text{ for }i\in V\text{ and }\delta_{i}=0\text{ for }i\not\in V, \end{array}
&  \text{ otherwise.}\end{array}\right.$$
Note that $T=\sum_{V\subset [c_n]} T|_V$.
\qed\end{defn}

It is easy to see that
$$F(z_{n-2})=\sum_{\delta_1=0}^1\cdots\sum_{\delta_{c_{n-1}}=0}^1\left(\begin{array}{ccccccc} \text{ }&1 & -1 & -1 & \cdots & -1& -1\\ 0&1& 1-2\delta_1 & 1-2\delta_2 & \cdots & 1-2\delta_{c_{n-1}} & \text{ }\end{array} \right).$$

\begin{defn}\label{def_of_tildeF2}
Let $$w={F}(z_{n-1})|_V= \left(\begin{array}{ccccccc} \text{ }&1 & -1 & -1 & \cdots & -1& -1\\ 0&1&1-2\delta_{1} & 1-2\delta_{2} & \cdots & 1-2\delta_{c_n} & \text{ }\end{array} \right),$$
where $\delta_{i}=1$ for $i\in V$ and $\delta_{i}=0$ for $i\not\in V$. We will define $\tilde{F}(w)$, which is a  modification of taking $F$, where we do not allow $(1+y^2)^{-1}$ to appear. In fact, if none of $1-2\delta_{i}$ is negative, then $\tilde{F}(w)=F(w)$. In general, $\tilde{F}(w)$ can be defined by letting
$$\tilde{F}(w)=\tilde{F}\left(\begin{array}{ccc} \text{ }&1 & -1 \\ 0&1& 1-2\delta_{1} \end{array} \right)\left[\prod_{i=2}^{c_n} \tilde{F}\left(\begin{array}{c} -1\\1-2\delta_{i} \end{array} \right)\right]\left(\begin{array}{c} -1\\ \text{ } \end{array} \right)$$
and using the following rules:
$$\tilde{F}\left(\begin{array}{ccc} \text{ }&1 & -1\\ 0&1 & 1-2\delta_1  \end{array} \right):=\sum_{\delta^{(1)}=0}^{1-\delta_{1}}\sum_{\delta^{(2)}=0}^{\Delta_{V,1}(1-\delta_{1})}\left(\begin{array}{cccc} \text{ }&1& -1& 2\delta_1-1\\ 0& 1& 2(1-\delta^{(1)})-1 & 2\Delta_{V,1}(1-\delta^{(2)})-1 \end{array}\right),$$
$$
\tilde{F}\left(\begin{array}{c} -1\\1-\delta_{i}r \end{array} \right):= \sum_{\delta=0}^{\Delta_{V,i}(1-\delta_{i})}\left( \begin{array}{c}  2\delta_i-1 \\  2\Delta_{V,i}(1-\delta) -1\end{array} \right) \,\,\,\,\,\,\,\,\, \text{ for }1<i\leq c_n,
$$
where 
$\Delta_{V,i}=\left\{\begin{array}{ll} 0, &\text{ if }\delta_i=0 \text{ and }\delta_{i+1}=1\\ 
1, & \text{ otherwise}.    \end{array} \right.$
\end{defn}

\begin{thm}\label{mainthm2}
We have
\begin{equation}\label{maineq2}\aligned x_n&=\sum_{V\subset [n-2]} \left(\sum_{W\subset [n-1]\setminus f(V)} \tilde{F}({F}(z_{n-2})|_V)|_W\right)=\sum_{V\subset [n-2]} \left( \tilde{F}({F}(z_{n-2})|_V)\right)\\&=\sum_{\alpha_i,\beta_i} \left(\begin{array}{ccccccc} \text{ }&1 &  \alpha_1 & \alpha_2 & \cdots & \alpha_{n-1}& -1\\ 0&1& \beta_{1} & \beta_{2} & \cdots & \beta_{n-1} & \text{ }\end{array} \right), \endaligned \end{equation}
where the summation runs over $\alpha_i,\beta_i$ satisfying the following:\\
$$\aligned \alpha_i=1\text{ or }&-1, \,\, \beta_i=1\text{ or }-1,\\&\alpha_1=-1,\\ \text{ if }\alpha_i=1&\text{ then }\beta_i=-1,\\ \text{ and if }\alpha_i=-1\text{ and }&\alpha_{i+1}=1\text{ then }\beta_i=-1. \endaligned$$
\end{thm}

Let us compare our formula with the known formula by Caldero and Zelevinsky \cite{CZ}. Let $j(U)$ denote the number of connected components of a subset $U\subset [n-1]\setminus\{1\}$, i.e.  $U=\cup_{i=1}^{j(U)}\{e_i,e_i+1,\cdots, e_i+l_i-1\}$ with $l_i>0$ $(1\leq i\leq j(U))$ and $e_i+l_i<e_{i+1}$ $(1\leq i\leq j(U)-1)$. In \cite[Proposition 4.3]{CZ}, it is shown that the number of $p$-element subsets $U\subset [n-1]\setminus\{1\}$ with $j(U)=t$ is equal to ${{p-1}\choose{t-1}}{{n-1-p}\choose{t}}$. For each $U\subset [n-1]\setminus\{1\}$, if we let $\alpha_i=1$ for $i\in U$ and $\alpha_i=-1$ for $i\not\in U$, then $|U|(=p)$ determines the degree of $x$ in the commutative case. The degree of $x$ is equal to $2p-n+1$. Let $t=j(U)$. If $U=\emptyset$ then $\beta_i=1$ or $-1$ for any $i\in[n-1]$.  If $U\neq\emptyset$ then only for $i\in [n-1]\setminus (U\cup \{e_i-1\}_{1\leq i\leq t})$, we have choices, namely $\beta_i=1$ or $-1$.  If $|\{i\,|\,\beta_i=1\}|=q$ then the degree of $y$ is equal to $2q-n+2$. Since $|[n-1]\setminus (U\cup \{e_i-1\}_{1\leq i\leq t})|= n-1-p-t$ for $U\neq\emptyset$, the specialization of (\ref{maineq2}) at $xy=yx$ is equal to
$$\tiny{\aligned
&\sum_q {{n-1}\choose q}x^{-n+1}y^{2q-n+2} + \sum_{p\geq 1} \sum_{q,t} {{n-1-p-t}\choose {q}}{{p-1}\choose{t-1}}{{n-1-p}\choose{t}}x^{2p-n+1}y^{2q-n+2}\\
&=\sum_q {{n-1}\choose q}x^{-n+1}y^{2q-n+2} +\sum_{p\geq 1} \sum_{q,t}{{n-1-p-t}\choose {q}}{{n-1-p}\choose{n-1-p-t}}{{p-1}\choose{t-1}}x^{2p-n+1}y^{2q-n+2}\\
&=\sum_q {{n-1}\choose q}x^{-n+1}y^{2q-n+2} +\sum_{p\geq 1} \sum_{q,t}{{n-1-p-q}\choose {n-1-p-t-q}}{{n-1-p}\choose{q}}{{p-1}\choose{t-1}}x^{2p-n+1}y^{2q-n+2}\\
&=\sum_q {{n-1}\choose q}x^{-n+1}y^{2q-n+2} +\sum_{p\geq 1} \sum_{q}{{n-1-p}\choose{q}}\sum_t {{n-1-p-q}\choose {n-1-p-t-q}}{{p-1}\choose{t-1}}x^{2p-n+1}y^{2q-n+2}\\
&=\sum_q {{n-1}\choose q}x^{-n+1}y^{2q-n+2} +\sum_{p\geq 1} \sum_{q}{{n-1-p}\choose{q}} {{n-2-q}\choose {n-2-p-q}}x^{2p-n+1}y^{2q-n+2}\\
&=\sum_q {{n-1}\choose q}x^{-n+1}y^{2q-n+2} +\sum_{p\geq 1} \sum_{q} {{n-2-q}\choose {p}}{{n-1-p}\choose{q}}x^{2p-n+1}y^{2q-n+2},
\endaligned}$$
which is precisely the same as \cite[Theorem 4.1 (4.3)]{CZ}.

\end{document}